\documentclass{amsart}
\usepackage{amssymb,amsmath,epsfig,mathrsfs, enumerate}
\usepackage[normalem]{ulem}
\usepackage{graphicx}
\usepackage[normalem]{ulem}
\usepackage[margin=3cm]{geometry}
\usepackage{hyperref}
 \usepackage[pagewise]{lineno}
\hypersetup{
 colorlinks   = true,
 urlcolor     = blue,
 citecolor   = red ,
 bookmarksopen=true
}
\newtheorem{theorem}{Theorem}[section]
\theoremstyle{plain}

\newtheorem{lemma}[theorem]{Lemma}

\newtheorem{proposition}[theorem]{Proposition}
\newtheorem{remark}[theorem]{Remark}

\numberwithin{equation}{section}
\title[]{Optimal harvesting for a logistic model with grazing}

\author{Mohan Mallick}
\address[Mohan Mallick]{VNIT Nagpur, India-440010} \email{mohan.math09@gmail.com, mohanmallick@mth.vnit.ac.in}
\author{Ardra A}
\address[Ardra A]{Department of Mathematics, IIT Palakkad, Kerala-678557, India} \email{ardra.math@gmail.com, 211814001@smail.iitpkd.ac.in}
\author{Sarath Sasi}
\address[Sarath Sasi]{Department of Mathematics, IIT Palakkad, Kerala-678557, India} \email{sarath@iitpkd.ac.in}
\date{}
\begin{document}
\maketitle
\begin{abstract}
	We consider  semi-linear elliptic equations of the following form:
\begin{equation*}
\left\{
\begin{aligned}
            -\Delta u &= \lambda[u-\dfrac{u^2}{K}-c \dfrac{u^2}{1+u^2}-h(x) u]=:\lambda f_h(u), \quad && x \in \Omega,\\
            \frac{\partial u}{\partial \eta}&+qu = 0, \quad && x\in\partial\Omega,
\end{aligned}
\right.
\end{equation*}
 where, $h\in U=\{h\in L^2(\Omega): 0\leq h(x)\leq H\}.$ We prove the existence and uniqueness of the positive solution for large $\lambda.$ Further, we establish the existence of an optimal control $h\in U$ that maximizes the functional
$J(h)=\int_{\Omega}h(x)u_h(x)~\rm{d}x-\int_{\Omega}(B_1+B_2 h(x))h(x)~\rm{d}x$
over $U$, where $u_h$ is the unique positive solution of the above problem associated with $h$,  $B_1>0$ is the cost per unit effort when the level of effort is low and $B_2>0$ represents the rate at which the cost rises as more labor is employed. Finally, we provide a unique optimality system.
\end{abstract}
 {\bf Keywor~\rm{d}s:}
 Optimal control, Spatial ecology,  Elliptic equations, Existence and uniqueness, Grazing and harvesting
\section{Introduction}\label{s1}
\noindent Let $\Omega\subset \mathbb{R}^N$ be a bounded domain with $C^2$ boundary. We study an optimal control problem for a nonlinear elliptic equation of the form
\begin{equation}
\label{main}
            \left\{
            \begin{split}
            -\Delta u &= \lambda[u-\dfrac{u^2}{K}-c \dfrac{u^2}{1+u^2}-h(x) u]=:\lambda f_h(u), ~~&&\rm{in}~\Omega,\\
            \frac{\partial u}{\partial \eta}&+qu = 0, ~~~\rm{on} ~\partial\Omega,
            \end{split} \right.
\end{equation}
where $\frac{1}{\lambda}$ is the diffusion coefficient, $K$, $c$ and $q$ are positive constants. Here $u$ is
the population density and $u-\dfrac{u^2}{K}$ represents logistics growth and the control $h$ represents the harvesting effort.
This model describes the grazing of a fixed number of grazers on a
logistically growing species (see \cite{May}-\cite{NM}). The rate of grazing is
given by $\dfrac{c u^2}{1+u^2}$ and the grazing population is assumed to be a constant. The model has also been used to describe the effect
of natural predators on fish populations. In such cases the term $\dfrac{cu^2}
{1+u^2}$ corresponds to natural predation. For more details see  \cite{May}, \cite{JS}, \cite{SH} and \cite{VS}. Robin boundary condition, where the flux at the boundary is proportional to the fish stock density, describes a scenario more favorable to the fish stock especially compared to the Dirichlet boundary condition in which the region surrounding our spatial domain is assumed to be lethal.
 
 \noindent The optimal harvesting problem for a population described by logistic growth was studied by Ca\~nada et al. in \cite{Canada98}. They studied \eqref{main} with  Dirichlet boundary condition taking $c=0$ and derived an optimality system that maximizes the payoff functional
$$ J(h)=\int_{\Omega}h(x)u_h(x)~\rm{d}x-\int_{\Omega} (h(x))^2~\rm{d}x,$$
where $u_h$ is the positive solution of \eqref{main}. The same logistic growth model  was studied by Ding and Lenhart in \cite{D-Lenhart2009}, where they discuss the existence and characterization of a $h$ which maximizes the functional
\begin{equation}
\label{obj}
J(h)=\int_{\Omega}h(x)u_h(x)~\rm{d}x-\int_{\Omega}(B_1+B_2 h(x))h(x)~\rm{d}x,
\end{equation}
where  $B_1>0$ is the cost per unit effort when the level of effort is low and $B_2>0$ represents the rate at which
the cost rises as more labor is employed. In this article we study the optimal control problem for the `logistic growth with grazing' model \eqref{main} with the objective of maximizing the payoff functional \eqref{obj}. The set of admissible controls is defined by
$$U=\{h(x)\in L^2(\Omega): 0\leq h(x)\leq H\},$$
where $0<H<1$ is a constant. We pose the optimal control problem in the setting where \eqref{main} has a unique positive solution.

The existence of an optimality system can be proved using the standard arguments given in  \cite{Canada98} and  \cite{D-Lenhart2009}. However, derivation of the optimality system and uniqueness becomes very challenging because of the model's nonlinear grazing term and the Robin boundary condition. Certain monotonicity arguments in \cite{Canada98}, which were also used in \cite{D-Lenhart2009}, become unfeasible because of the grazing term. The Robin boundary condition and the new solution space introduce new challenges in obtaining certain estimates. \\

\noindent First, we state a result that gives the existence and uniqueness of positive solutions of \eqref{main}.
\begin{theorem}\label{T1}
For $c<2(1-H)$, there exists a $K^*>0$ such that for $K>K^*$ \eqref{main} has a unique positive solution for large $\lambda$.
\end{theorem}
\noindent Considering the optimal control problem in the above setting we have the following existence result for the optimal control.
 \begin{theorem}\label{T2}
 Let $c$, $K^*$ be as in Theorem \ref{T1}. Then for $K> K^*$ there exists $h^*\in U$ that maximizes $J(h)$.
\end{theorem}
\noindent In the next theorem, we give a characterization for any optimal $h$.
 \begin{theorem}\label{T3}
Let $c<2(1-H)$ and both $K$, $\lambda$ be large. Also, let $h\in U$ be any optimal control and $u_h$ be the corresponding maximal solution. Then  $h(x)$ is characterized by,
\begin{equation}\label{optcond}
 h(x)=\min\Big\{H,\max\{0,\frac{u_h-pu_h-B_1}{2B_2}\}\Big\},
\end{equation}
where $p$ in $H^2(\Omega)$ is a solution to the adjoint problem
 \begin{equation}
\left\{
   \begin{split}\label{eq:adjoint}
  &-\Delta p -\lambda \Big(p-\frac{2u_hp}{K}-2c\frac{u_h}{(1+u_h^2)^2} p-h(x)p \Big)=h(x),~~\rm{in}~\Omega,\\
  &\frac{\partial p}{\partial\eta}+qp=0, ~~~\rm{on} ~\partial\Omega.
  \end{split}
   \right.
 \end{equation}
\end{theorem}
\noindent In the case $B_2>0,$ the state equation \eqref{main} and the adjoint equation \eqref{eq:adjoint}  together with \eqref{optcond} is called optimality system (OS), which is given by 
\begin{equation}\label{os}
\left\{
\begin{split}
&-\Delta u_h=\lambda\left[u_h-\frac{u_h^2}{K}-\frac{cu_h^2}{1+u_h^2}-h(x)u_h\right],~~\rm{in}~\Omega,\\
&\frac{\partial u_h}{\partial\eta}+qu_h=0,~~~\rm{on} ~\partial\Omega,\\
&-\Delta p-\lambda\left[p-\frac{2u_hp}{K}-\frac{2cu_hp}{(1+u_h^2)^2}-hp\right]=h,\rm{in}~\Omega,\\
&\frac{\partial p}{\partial\eta}+qp=0,~~~\rm{on}~\partial\Omega.
\end{split}\right.
\end{equation}
Next, we state a uniqueness result which gives a characterization of the unique optimal control in terms of the unique solutions of (OS).
\begin{theorem}\label{T4}
Let $N\in\{2,3\}$. The solution $h,p,u_h$ of the optimality system (OS), with $u_h>0$ in $\Omega$, is unique for a large value of $B_2$.
\end{theorem}

\noindent In the next section, we will prove Theorem \ref{T1}. In Section \ref{s3} we will prove
the existence of an optimal control. The optimality system will be derived in Section \ref{s4} and finally in Section \ref{s5} we will prove the uniqueness of the optimality system i.e., Theorem \ref{T4}.

\section{Existence and uniqueness of positive solutions}\label{s2}
\noindent Consider the following autonomous problem:
\begin{equation}
\label{general}
            \left\{
            \begin{split}
            -\Delta u &= \lambda f(u), &&\quad x \in \Omega, \\
           \frac{\partial u}{\partial \eta}+&q u = 0, &&\quad x \in\partial\Omega.
            \end{split} \right.
\end{equation}
The existence of positive solutions of \eqref{general} has been proved under various assumptions on $f$ (see \cite{mohan} for a discussion). 
We are interested in a class of nonlinearities satisfying the following hypothesis:\\
{\bf (A):}  $f\in C^2([0,\infty))$ is such that $f(0)=0$, $f'(0)>0$, and there exists $r_0>0$ with  $f(s)>0$ on $(0,r_0)$ and $f(s)<0$ for $s>r_0$.
It is known that for such $f$, the boundary value problem, \eqref{general} admits a positive solution for $\lambda>\frac{\lambda_1(\Omega)}{f'(0)}$, where $\lambda_1(\Omega)$ is the principal eigenvalue of the following eigenvalue problem:
\begin{equation}
\label{eigen}
            \left\{
            \begin{split}
            &-\Delta u = \lambda u, \quad x \in \Omega, \\
           &\frac{\partial u}{\partial \eta}+q u = 0, \quad x \in\partial\Omega.
            \end{split} \right.
\end{equation}
See Theorem 1.1 in Mohan et al \cite{mohan} for a proof for the same. 
In the same paper by  Mohan et al., they have proved the following uniqueness result: 
 \begin{proposition}
 	\label{Dancer}
 Let $f$ satisfy (A), then the following results hold:
\begin{enumerate}[(i)]
\item For all sufficiently large $\lambda$, the boundary value problem \eqref{general} has a unique positive solution $u$
such that $\|u\|_\infty\leq r_0 $ (see Theorem 1.3 of \cite{mohan}).
\item Moreover, given $\delta>0$ there exists $\lambda_\delta>0$ such that for $\lambda>\lambda_\delta$ the positive solution of \eqref{general} satisfies the following: $$r_0-\frac{\delta}{2}<u(x)\leq r_0 \mbox{ for } x\in \bar{\Omega}$$ (see Lemma 3.1 of \cite{mohan}).
\end{enumerate}
\end{proposition}
\noindent The following result will follow from the above discussion.
\begin{proposition}
	\label{existence_fH}
  Let $0\leq \alpha\leq H$, $c<2(1-\alpha)$, and $f_{\alpha}(s)=s-\frac{s^2}{K}-c\frac{s^2}{1+s^2}-\alpha s$. Then there exists $K^*_{\alpha}>0$ and $\lambda_{\alpha}>0$ such that for $K>K^*_{\alpha}$ and $\lambda>\lambda_{\alpha}$, the boundary value problem
  \begin{equation}
\label{fH}
            \left\{
            \begin{split}
            &-\Delta u = \lambda f_{\alpha}(u), \quad x \in \Omega\\ 
            &\frac{\partial u}{\partial \eta}+q u = 0, \quad x \in\partial\Omega
            \end{split} \right.
\end{equation}
has a unique positive solution $u_{\alpha}$.
\end{proposition}
\begin{proof}
By Proposition \ref{Dancer}, it is enough if we show that $f_{\alpha}$ satisfies (A).
Using arguments similar to the ones made in the proof of Proposition 3.2 of \cite{LSS}, it is easy to show that for $c<2(1-{\alpha})$ there exists a $K^*_{\alpha}$ such that for $K>K^*_{\alpha}$ there is a unique $r_0>0$ such that $f_{\alpha}(r_0)=0$. 
\end{proof}
\begin{remark}\label{r01}
{\rm Note that $f_H(\frac{K(1-\alpha)}{2})>0$ and $f_{\alpha}(K)<0$ and hence $r_0\in(\frac{K(1-\alpha)}{2}, K)$.}
\end{remark}

\noindent Let $c<2(1-H).$ Corresponding to $\alpha=0$, $\alpha=H$, there exist $\lambda\geq \max\{\lambda_0,\lambda_H\}$ and $K^*\geq\max\{K^*_0,K^*_H\}$ such that \eqref{general} has unique solutions $u_0$ and $u_H$ respectively. The next proposition gives an estimate for $u_0$ and $u_H$.
\begin{proposition}\label{remark2}
 Let $K>K^*$ and $c<2(1-H)$. Given $\delta>0$ there exists $\lambda_{\delta}>\max\{\lambda_0,\lambda_H\}$ such that for $\lambda>\lambda_{\delta}$, $\frac{K(1-H)}{2}<u_H<u_0<K$ in $\bar{\Omega}$.
\end{proposition}
\begin{proof}
Clearly $\|u_0\|_\infty<K$ as $f_0(K)<0$ for $K>0$. Since $f_H(s)<f_0(s)$ for all $s>0$, we have $u_H< u_0$. Let $\delta>0$ be such that $\frac{K(1-H)}{2}<r_0-\frac{\delta}{2}$. Then by Proposition 2.1  there exists a $\lambda_{\delta}$ such that for $\lambda>\lambda_{\delta}$,
 $\frac{K(1-H)}{2}<u_{H}< u_{0}<K$ in $\bar{\Omega}.$
\end{proof}
\noindent Now we will prove the existence and uniqueness of positive solutions of the non-autonomous equation \eqref{main}. We will use the method of sub-supersolutions to prove our existence result. By a subsolution (Supersolution) of \eqref{main}, we mean a function $\psi\in H^1(\Omega)\cap C(\bar{\Omega})$ with
 \begin{equation*}\label{Subsup}
\int_{\Omega}\nabla\psi\nabla\phi ~\rm{d}x +\int_{\partial\Omega}q\psi \phi ~\rm{d}s\leq(\geq) \int_{\Omega}\lambda[\psi-\frac{\psi^2}{K}-c\frac{\psi^2}{1+\psi^2}-h(x)\psi]\phi ~\rm{d}x \quad \rm{in} ~\Omega,
\end{equation*}
for every $\phi\in  C^{\infty}(\Omega)$ with $\phi\geq 0$ in $\Omega$. Then the following lemma holds (see \cite{Am})
\begin{lemma}\label{lemmasub}
  Let $\psi$ and $\phi$ be a subsolution and supersolution of \ref{main} respectively. Then \ref{main} has a solution $u\in H^1(\Omega)$ such that $\psi \leq u \leq \phi$. 
\end{lemma}
\noindent {\bf Proof of Theorem \ref{T1}:}\\
\noindent {\bf Existence:} In the proof of Proposition \ref{existence_fH}, we have seen that there exists a $K^*>0$  such that for $K>K^*$ the function $f_{H}(s)=s-\frac{s^2}{K}-c\frac{s^2}{1+s^2}-Hs$ satisfies the hypothesis (A). Hence for $\lambda>\dfrac{\lambda_1(\Omega)}{f_{H}'(0)}=\dfrac{\lambda_1(\Omega)}{1-H}$, the boundary value problem \eqref{fH} has a positive solution $u_{H}$. It is easy to check that $u_{H}$ is a subsolution for \eqref{main} and $\Psi\equiv K$ is a supersolution for \eqref{main}. Hence by Lemma \ref{lemmasub}, \eqref{main} has a positive solution for $\lambda>\frac{\lambda_1(\Omega)}{1-H}$.

\noindent {\bf Uniqueness:} We will show that \eqref{main} has a unique positive solution for large $\lambda$. Let $u_h$ be any positive solution of \eqref{main}, then $u_h$ is a super solution to \eqref{fH}. By Proposition \ref{existence_fH}, we have  $u_{H}$ is the unique solution of \eqref{fH} for large values of the parameter $\lambda$. Thus $u_{H}\leq u_h$ for large $\lambda$. A similar argument can be used to prove that $u_h\leq u_0$
for $\lambda\gg 1$. Hence for $\lambda\gg 1$, any positive solutions $u_h$ of \eqref{main} satisfies
\begin{equation*}\label{main3}
u_{H}\leq u_h\leq u_0.
\end{equation*}
Choose $\delta>0$ such that $\frac{K(1-H)}{2}<r_0-\frac{\delta}{2}$, where $r_0$ is such that $f_H(r_0)=0$. Then by Proposition \ref{remark2}, there exists a $\lambda_{\delta}$ such that for $\lambda>\lambda_{\delta}$,
\begin{eqnarray}\label{123}
\frac{K(1-H)}{2}\leq u_{H}\leq u_{h}\leq u_{0}<K~~~~\rm{in}~~ \Omega.
\end{eqnarray}
Let $u_{m}$ be the maximal positive solution of \eqref{main} and $u_h$ be any positive solution of \eqref{main} different from $u_m$, then
\begin{align}\label{compare}
u_m\Delta u_h-u_h\Delta u_m
=\lambda[u_hf_h(u_m)-u_mf_h(u_h)]
=\lambda u_hu_{m}\left[\frac{f_{h}(u_m)}{u_m}-\frac{f_{h}(u_{h})}{u_{h}}\right].
\end{align}
Integrating \eqref{compare} over $\Omega$ we get
\begin{align*}
&\int_{\Omega}[u_{m}\Delta u_h-u_h\Delta u_{m}]~\rm{d}x=\lambda\int_{\Omega}u_hu_{m}\left[\frac{f_{h}(u_m)}{u_m}-\frac{f_{h}(u_{h})}{u_{h}}\right]~\rm{d}x\\\nonumber
&\leq \lambda K^2\int_{\Omega}\left[\frac{f_{h}(u_m)}{u_m}-\frac{f_{h}(u_{h})}{u_{h}}\right]~\rm{d}x<0.
\end{align*}
But
\begin{align*}
\int_{\Omega}[u_m\Delta u_h-u_h\Delta u_m]~\rm{d}x&=\int_{\Omega}\nabla u_h\cdot\nabla u_m~\rm{d}x-\int_{\Omega}\nabla u_h\cdot\nabla u_m~\rm{d}x-\int_{\partial\Omega}u_h\frac{\partial u_m}{\partial\eta}~\rm{d}s+\int_{\partial\Omega}u_m\frac{\partial u_h}{\partial\eta}~\rm{d}s\\
&=q\int_{\partial\Omega}u_hu_m~\rm{d}s-q\int_{\partial\Omega}u_hu_m~\rm{d}s=0
\end{align*}
This is a contradiction. Hence for $\lambda\gg 1$, \eqref{main} has a unique positive solution.
\qed
\section{Existence of an optimal control}\label{s3}
\noindent In this section, we prove the existence of an optimal control for our objective functional.

\noindent {\bf Proof of Theorem 1.2.}

 \noindent We have
 $$J(h)=\int_{\Omega}h(x)u(x)~\rm{d}x-\int_{\Omega}(B_1+B_2 h)h ~\rm{d}x\leq \|u\|_\infty H |\Omega|.$$
Let $\{h_n\}\subset U$ be any maximizing sequence for $J$. Consequently $\displaystyle{\lim_{n\to \infty}}J(h_n)=\displaystyle{\sup_{h\in U}} J(h)$.
Since $u_{h_n}$ is a solution of (\ref{main}) with $h=h_n$ we have
\begin{equation}\label{3.1}
  \int_{\Omega}\nabla u_{h_n}\cdot \nabla v~\rm{d}x+q\int_{\partial\Omega}u_{h_n}v~\rm{d}s=\lambda\int_{\Omega}\Big(u_{h_n}-\frac{u_{h_n}^2}{K}-c\frac{u_{h_n}^2}{1+u_{h_n}^2}-h_n u_{h_n}\Big)v~\rm{d}x.\\
\end{equation}
Taking $v=u_{h_n}$ in \eqref{3.1} and since $q>0$ we have 
\begin{eqnarray*}
 \int_{\Omega}|\nabla u_{h_n}|^2~\rm{d}x\leq \int_{\Omega}|\nabla u_{h_n}|^2~\rm{d}x +q\int_{\partial\Omega}u_{h_n}^2~\rm{d}s&=&\lambda\int_{\Omega}\Big(u_{h_n}-\frac{u_{h_n}^2}{K}-c\frac{u_{h_n}^2}{1+u_{h_n}^2}-h_n u_{h_n}\Big)u_{h_n}~\rm{d}x\\
 &\leq& \lambda \int_{\Omega}u_{h_n}^2~\rm{d}x\leq\lambda K^2|\Omega|,
\end{eqnarray*}
and hence $\|u_{h_n}\|_{H^1(\Omega)}\leq C$ for some constant $C>0$.
Thus there exists $u^*\in{H^1(\Omega)}$ such that a subsequence, which we denote by ${u_{h_n}}$ itself, converges weakly to
$u^*$ in ${H^1(\Omega)}$. Further $u_{h_n}\rightarrow u^*$ strongly in $L^2(\Omega)$ as ${H^1(\Omega)} \subset\subset L^2(\Omega)$. This subsequence can be
chosen such that the convergence is almost uniform. Since $\{h_n\}\subset U$ is uniformly bounded in $L^2(\Omega)$, $h_n$ converges weakly to some $h^*\in U$. Also, a simple integration by parts argument gives us $u_{h_n}\rightharpoonup u^*$ in $L^2(\partial\Omega)$.

\noindent Now we will prove that $u^*=u_{h^*}$. It is enough if we can show that
$$ \int_{\Omega}\nabla u^*\cdot \nabla v~\rm{d}x+\int_{\partial\Omega}qu^*v~\rm{d}s=\lambda\int_{\Omega}\Big(u^*-\frac{(u^*)^2}{K}-c\frac{(u^*)^2}{1+(u^*)^2}-h^* u^*\Big)v~\rm{d}x,\quad v\in{H^1(\Omega)}. $$
Since $u_{h_n}\rightharpoonup u^*$ in ${H^1(\Omega)}$, we get
\begin{equation*}
  \int_{\Omega}\left[\nabla u_{h_n}\cdot \nabla v-\nabla u^*\cdot \nabla v\right]~\rm{d}x= \int_{\Omega}(\nabla u_{h_n}-\nabla u^*)\cdot\nabla v~\rm{d}x\to 0.
\end{equation*}
Also as $u_{h_n}\rightharpoonup u^*$ in $L^2(\partial\Omega),$ 
\begin{eqnarray*}
     \int_{\partial\Omega}\left[ qu_{h_n}v-qu^*v \right] ~\rm{d}s&\leq& \Big(\int_{\partial \Omega}|(u_{h_n}-u^*)^2| ~\rm{d}s \Big)^{\frac{1}{2}}  \Big(\int_{\partial \Omega}|qv|^2 ~\rm{d}x\Big)^\frac{1}{2} ~\rm{d}s\\
 &\to&0 \mbox{~~as~} n\to\infty. 
\end{eqnarray*}
Now we will show that
$$\int_{\Omega}|\Big(u_{h_n}-\frac{u_{h_n}^2}{K}-c\frac{u_{h_n}^2}{1+u_{h_n}^2}-h_n u_{h_n}\Big)v-
\Big(u^*-\frac{(u^*)^2}{K}-c\frac{(u^*)^2}{1+(u^*)^2}-h^* u^*\Big)v|~\rm{d}x\to 0$$
as $n\to\infty$. We do this by grouping together
similar terms and showing that each converges to $0$ as $n\to\infty$.
\noindent As $u_{h_n}\to u^*$ in $L^2(\Omega)$ we get the following:
\begin{eqnarray*}
 \int_{\Omega}|(u_{h_n}-u^*)v| ~\rm{d}x&\leq& \Big(\int_{\Omega}|(u_{h_n}-u^*)^2| ~\rm{d}x \Big)^{\frac{1}{2}}  \Big(\int_{\Omega}|v|^2 ~\rm{d}x\Big)^\frac{1}{2} ~\rm{d}x\\
 &\to&0 \mbox{~~as~} n\to\infty.
\end{eqnarray*}
We can use this to further show that $\int_{\Omega}|(u_{h_n}^2-{u^*}^2)v| ~\rm{d}x\to 0$
as $n\to\infty$ and
\begin{eqnarray*}
\int_{\Omega}\Big| \Big(\frac{u_{h_n}^2}{1+u_{h_n}^2}-\frac{{u^*}^2}{1+{u^*}^2}\Big)v \Big|~\rm{d}x&=& \int_{\Omega}
\Big| \frac{(u_{h_n}^2-{u^*}^2)}{(1+u_{h_n}^2)(1+{u^*}^2)}v\Big|~\rm{d}x\\
&\leq&\int_{\Omega}|(u_{h_n}^2-{u^*}^2)v|~\rm{d}x\\
 &\to&0 \mbox{~~as~} n\to\infty.
\end{eqnarray*}
Finally
\begin{eqnarray*}
 |\int_{\Omega}\Big(h_n u_{h_n}-h^*u^*\Big)v| ~\rm{d}x&\leq&  |\int_{\Omega}h_n \Big(u_{h_n}-u^*\Big)v ~\rm{d}x|+
|\int_{\Omega}\Big(h_n -h^*\Big)u^*v~\rm{d}x|\\
&\leq&\int_{\Omega}H |\Big(u_{h_n}-u^*\Big)v| ~\rm{d}x+
|\int_{\Omega}\Big(h_n -h^*\Big)u^*v~\rm{d}x|\\
 &\to&0 \mbox{~~as~} n\to\infty
\end{eqnarray*}
Now we will prove that $h^*$ maximizes the functional $J(h)$ by showing that $\displaystyle{\sup_{h\in U}} J(h)\leq J(h^*)$. 
\begin{eqnarray*}
\displaystyle{\sup_{h\in U}} J(h)&=& \displaystyle{\lim_{n\to\infty}}J(h_n)\\
&=&\displaystyle{\lim_{n\to\infty}}\left(\int_{\Omega}h_n(x)u_{h_n}(x)~\rm{d}x-\int_{\Omega}(B_1+B_2 h_n)h_n ~\rm{d}x\right)\\
&\leq&\int_{\Omega}h^*u_{h^*}~\rm{d}x-\underline{\lim}_{n\to\infty}\int_{\Omega}(B_1+B_2 h_n)h_n~\rm{d}x\\
&\leq&\int_{\Omega}h^*u_{h^*}~\rm{d}x-\int_{\Omega}(B_1+B_2 h^*)h^*~\rm{d}x\\
&=&J(h^*).
\end{eqnarray*}
Thus we have shown that there exists a control $h^*$ which maximizes $J(h)$.\qed
\section{The Optimality System}\label{s4}
In this section, we deduce the optimality system for the functional $J$. We present some preliminary results which will be used in the sequel.

Consider the  eigenvalue  problem
\begin{equation}\label{eig-q}
  \begin{split}
  -\Delta u(x)+V(x)u(x) &=\sigma u(x) \quad \rm{in} ~\Omega,\\
  \frac{\partial u}{\partial\eta}+qu(x)&=0 \quad \rm{on}~ \Omega,
  \end{split}
\end{equation}
where $V(x)\in L^\infty(\Omega)$. 
The principal eigenvalue $\sigma_1(V)$ is characterized  by
\begin{equation*}
 \sigma_1 (V) =\inf_{\phi\in{H^1(\Omega)}\setminus \{0\}}\dfrac{\int_{\Omega}|\nabla \phi|^2~\rm{d}x+\int_{\Omega}V\phi^2~\rm{d}x+\int_{\partial \Omega}q\phi^2~\rm{d}s}{\int_{\Omega}|\phi|^2~\rm{d}x}.
\end{equation*}
 It is  known that $\sigma_1(V)$ is simple and the associated eigenfunctions
$\phi_V\in C^{1,\alpha}(\overline{\Omega}),~\alpha\in(0,1)$ (see Proposition 2.4 in \cite{Fragnelli2016}). Let $\phi_V$ be the normalized eigenfunction with $\|\phi_V\|_\infty=1$ .
The principle eigenvalue $\sigma_1(V)$ has the following properties:
\begin{itemize}
    \item[1] $\sigma_1(V)$ is increasing with respect to $V$, that is, if $V_1<V_2$, then $\sigma_1(V_1)<\sigma_2(V_2)$.
    \item[2] $\sigma_1(V)$ is continuous with respect to $V\in L^{\infty}(\Omega)$.
 \item[3] If $\sigma_1(V)>0$, we can find a constant  $M>0$ such that
\begin{equation}
\label{eig-q1}
 M|\phi|_{H^1(\Omega)^2}\leq \int_{\Omega}|\nabla \phi|^2~\rm{d}x+\int_{\Omega}V\phi^2~\rm{d}x+\int_{\partial\Omega}q\phi^2~\rm{d}s\quad \forall \phi\in H^1(\Omega).
\end{equation}
\end{itemize}
Establishing 1 and 2 is a straightforward exercise. Here we present a proof of property 3, part of which follows along the lines of the discussion in \cite{Canada98}, where a similar result is proved for the Dirichlet boundary case. 

Let $0<C<\frac{\sigma_1(V)}{\sigma_1(V)+\|V\|_{\infty}}$ that is $C\|V\|_{\infty}\leq (1-c)\sigma_1(V)$, then by the definition of $\sigma_1(V)$, we have for all $\phi\in H^1(\Omega)$
\begin{align*}
    (1-C)\int_{\Omega}(|\nabla\phi|^2+V\phi^2)~\rm{d}x+(1-C)q\int_{\partial\Omega}\phi^2~\rm{d}s&\geq(1-C)\sigma_1(V)\int_{\Omega}\phi^2~\rm{d}x\\
    &\geq C\|V\|_{\infty}\int_{\Omega}\phi^2~\rm{d}x\\
    &\geq-C\int_{\Omega}V\phi^2~\rm{d}x.
\end{align*}
As $q>0$, we have 
$$(1-C)\int_{\Omega}\left(|\nabla \phi|^2+V\phi^2\right)~\rm{d}x+q\int_{\partial\Omega}\phi^2~\rm{d}s\geq -C\int_{\Omega}V\phi^2~\rm{d}x.$$
Rearranging the above equation we have 
\begin{equation}\label{Bounded below by C}
C\int_{\Omega}|\nabla \phi|^2~\rm{d}x\leq \int_{\Omega}|\nabla \phi|^2~\rm{d}x+\int_{\Omega}V\phi^2~\rm{d}x+\int_{\partial\Omega} q \phi^2 ~\rm{d}s,\quad \forall \phi\in{H^1(\Omega)}.
\end{equation}

From the variational characterization of $\sigma_1(V),$ it follows that 
\begin{equation}
    \label{Bounded below by sigma_1}
     \sigma_1(V) \int_{\Omega} \phi^{2} ~\rm{d}x \leq \int_{\Omega} |\nabla \phi |^{2} ~\rm{d}x + \int_{\Omega} V \phi^{2} ~\rm{d}x + \int_{\partial \Omega} q \phi^{2} ~\rm{d}s,  \quad \forall \phi\in{H^1(\Omega)}.
\end{equation}
Set $M:= \frac{1}{2}\displaystyle{\inf} \{ C, \sigma_1(V) \}$.
Adding \eqref{Bounded below by C} and \eqref{Bounded below by sigma_1}, we get
\begin{equation*}
    M  |\phi|_{H^1(\Omega)}^2 \leq \int_{\Omega}|\nabla \phi|^2 ~\rm{d}x+\int_{\Omega}V\phi^2 ~\rm{d}x+
    \int_{\partial\Omega}q\phi^2 ~\rm{d}s,  \quad \forall \phi\in H^1(\Omega).
\end{equation*}

Next, we state a useful existence and uniqueness result.
\begin{lemma}
\label{lin-q-unique}
Assume that 
 $V\in L^\infty(\Omega)$, and $\sigma_1(V)>0.$ Then for each $f\in L^2(\Omega)$ the linear problem
\begin{equation}\label{lin-q}
  \begin{split}
  -\Delta u(x)+V(x)u(x) &= f(x) \quad x\in\Omega,\\
  \frac{\partial u}{\partial\eta}+qu(x)&=0 \quad x\in \partial \Omega,
  \end{split}
\end{equation}
has a unique solution $u\in {H^1(\Omega)}$.
\end{lemma}
\begin{proof}   
Consider the bilinear form $B:H^1(\Omega) \times H^1(\Omega) \to \mathbb{R}$ defined by 
\begin{equation*}
    B[u,v] = \int_{\Omega} \nabla u \cdot \nabla v ~\rm{d}x +\int_{\Omega} V(x)uv ~\rm{d}x + \int_{\partial \Omega} q uv ~\rm{d}s. 
\end{equation*}
Then using H\"older's inequality, we have 
\begin{equation*}
    \begin{aligned}
         \left|B[u,v] \right| & \leq \left\|\nabla u \right\|_{L^2(\Omega)}\left\|\nabla v \right\|_{L^2(\Omega)}+\left\|V \right\|_\infty\left\|u \right\|_{L^2(\Omega)}\left\|v \right\|_{L^2(\Omega)}+q\left\|u \right\|_{L^2(\partial\Omega)}\left\|v \right\|_{L^2(\partial \Omega)} \\
         & \leq \left\|\nabla u \right\|_{L^2(\Omega)}\left\|\nabla v \right\|_{L^2(\Omega)}+\left\|V \right\|_\infty\left\|u \right\|_{L^2(\Omega)}\left\|v \right\|_{L^2(\Omega)}+qK\left\|u \right\|_{H^1(\Omega)}\left\|v \right\|_{H^1( \Omega)}, \text{ (using trace inequality)} \\
         & \leq  (1+ \left\|V\right\|_\infty+qK) \left\|u \right\|_{H^1(\Omega)}\left\|v \right\|_{H^1( \Omega)}.
    \end{aligned}
\end{equation*}
Also, by property 3 for $\sigma_1(V)$ above, we know, for all $u\in H^1(\Omega),$
\begin{equation*}
    M \left\|u \right\|_{H^1(\Omega)}^2 \leq B[u,u]. 
\end{equation*}
Let us define a bounded linear function $\phi : H^1(\Omega) \to \mathbb{R}$ by $\phi(v)= \int_{\Omega}  fv ~\rm{d}x.$ Then by the Lax Milgram theorem there exists a unique $u\in H^1(\Omega)$ such that $B[u,v]=\phi(v)$ for all $v \in H^1(\Omega).$
\end{proof}
 
\noindent In \cite{Canada98} and \cite{D-Lenhart2009}, the authors use the monotonicity of $\sigma_1(V)$ to obtain certain existence results while deriving the optimality system. The introduction of the grazing terms in the nonlinearity makes it impossible to employ a similar argument in our case. The next lemma will help us to arrive at the optimality system.

\begin{lemma}\label{newlemma}
Let $c<2(1-H)$, $K\gg 1$, and $u_g$, $u_h$ be solutions of \eqref{main} corresponding to $h,g\in U$ respectively. Then there exists $\lambda_{\delta}>0$ such that for $\lambda>\lambda_{\delta}$, the principal eigenvalue $\sigma_1\Big(\lambda\Big(-1+h(x)+\frac{u_h+u_g}{K}+c\frac{u_h+u_g}{(1+u_h^2)(1+u_{g}^2)}\Big)\Big)$ of \eqref{eig-q} is positive.
\end{lemma}
\begin{proof}
Since $u_h$ is a solution of \eqref{main}, we have
\begin{equation*}
-\Delta u_h+\lambda\left(-1+h(x)+\frac{u_h}{K}+c\frac{u_h}{1+u_h^2}\right)u_h=0.
\end{equation*}
This implies that $\sigma_1\big(\lambda(-1+h(x)+\frac{u_h}{K}+c\frac{u_h}{1+u_h^2})\big)=0$.
Let
\begin{align*}
&V_1=\lambda\Big(-1+h(x)+\frac{u_h}{K}+c\frac{u_{h}}{1+u_{h}^2}\Big),\\
&V_2=\lambda\Big(-1+h(x)+\frac{u_h+u_g}{K}+c\frac{u_h+u_g}{(1+u_h^2)(1+u_{g}^2)}\Big).
\end{align*}
Hence
\begin{align*}
V_2-V_1 &= \lambda\Big(\frac{u_g}{K}+c\frac{u_{h}+u_g-u_{h}(1+u_g^2)}{(1+u_g^2)(1+u_{h}^2)}\Big)\nonumber\\
&=\lambda u_g\Big(\frac{1}{K}+c\frac{(1-u_{h}u_g)}{(1+u_g^2)(1+u_{h}^2)}\Big).
\end{align*}
To estimate $V_2-V_1$ we define $g\colon \mathbb{R} \to \mathbb{R}$ as,
\begin{equation*}
g(x)= \frac{1}{x}+\frac{c(1-x^2)}{(1+x^2)^2}.
\end{equation*}
Let $x_{0}$ be the smallest positive number such that $g(x)>0$, for $x>x_{0}$. Choose $K>\bar{K}=\max\{x_{0},\frac{8}{(1-H)}\}$ and $\alpha_0=\frac{K(1-H)}{2}g(K)$. Let $\delta>0$ be such that $\frac{K(1-H)}{2}<r_0-\frac{\delta}{2}$. Then by \eqref{123}, for $\lambda>\lambda_{\delta}$ we have the following estimate for $V_2-V_1$ in $\Omega$,
\begin{align*}
V_2-V_1 &=\lambda u_g\Big(\frac{1}{K}+c\frac{(1-u_g u_{h})}{(1+u_g^2)(1+u_{h}^2)}\Big)\\\nonumber
&>\lambda\frac{K(1-H)}{2}\Big(\frac{1}{K}+\frac{c(1-K^2)}{(1+K^2)^2}\Big)\\
&=\lambda\frac{K(1-H)}{2}g(K)=\lambda\alpha_0.
\end{align*}
To show that $\sigma_1(V_2)>0$, we will first show that $\sigma_1(V_2)>\sigma_1(V_1)$. Now, $\sigma_1(V_2)$ has the following characterization 
\begin{align*}
\sigma_1 (V_2) =\inf_{\phi\in{H^1(\Omega)}\setminus \{0\}}\dfrac{\int_{\Omega}|\nabla \phi|^2~\rm{d}x+\int_{\Omega}V_2\phi^2~\rm{d}x+\int_{\partial\Omega}q\phi^2~\rm{d}s}{\int_{\Omega}|\phi|^2~\rm{d}x},
\end{align*}
Let $\phi_{V_2}$ be the corresponding eigenfunction with $\|\phi_{V_2}\|_{L^2}=1$. Then
\begin{align*}
\sigma_1(V_2)&=\int_{\Omega}|\nabla \phi_{V_2}|^2 ~\rm{d}x+ \int_{\Omega}(V_2-V_1)\phi_{V_2}^2~\rm{d}x+\int_{\Omega}V_1\phi_{V_2}^2~\rm{d}x+\int_{\partial\Omega}q\phi^2_{V_2}~\rm{d}s\\
&\geq{\sigma_1(V_1)}+ \int_{\Omega}(V_2-V_1)\phi_{V_2}^2~\rm{d}x.
\end{align*}
Thus
\begin{align}
\sigma_1(V_2)-\sigma_1(V_1)\geq\int_{\Omega}(V_2-V_1)\phi_{V_2}^2~\rm{d}x&=\int_{\Omega}(V_2-V_1)\phi_{V_2}^2~\rm{d}x
\geq \lambda\alpha_0\int_{\Omega}\phi_{V_2}^2~\rm{d}x>0.\label{124}
\end{align}
Hence $\sigma_1(V_2)>0$ as $\sigma_1(V_1)=0$.
\end{proof}
\noindent In order to characterize the optimal control we examine the differentiability of the mapping $h\to u_h$, whose derivative is called the sensitivity. We say that the mapping $h\mapsto u_h$
is differentiable at $h$ if there exists $\psi\in{H^1(\Omega)}$, such that
$$\frac{u_{h+\gamma\epsilon}-u_{h}}{\epsilon}\rightharpoonup \psi \mbox{~weakly in~${H^1(\Omega)}$ as~}\epsilon\to 0,$$
where $\gamma\in L^\infty(\Omega)$ with $h+\epsilon \gamma\in U$.
\begin{lemma}
Let $c<2(1-H)$ and $K>{K^*}$, then for $\lambda\geq \lambda_{\delta}$ there exists a $\psi\in{H^1(\Omega)}$ such that
$\frac{u_{h+\epsilon\gamma}-u_{h}}{\epsilon}\rightharpoonup \psi$  weakly in ${H^1(\Omega)}$ as $\epsilon \to 0$ for any $\gamma\in L^\infty(\Omega)$. Further $\psi$
is the unique solution of the linear problem
\begin{equation}\label{eqn:sensitivity}
\left\{
  \begin{split}
  -\Delta \psi &=\lambda \Big(\psi-\frac{2u_h}{K}\psi-2c\frac{u_h}{(1+u_h^2)^2}\psi-h\psi-\gamma u_h\Big)\quad \mbox{in~} \Omega,\\
  \frac{\partial \psi}{\partial\eta}&+q\psi=0 \quad \mbox{on~} \partial \Omega.
  \end{split}
  \right.
\end{equation}
\end{lemma}
\begin{proof}
We have
\begin{equation}\label{eq:uep}
\left\{
\begin{aligned}
 -\Delta u_{h+\epsilon\gamma}&=\lambda\left(u_{h+\epsilon\gamma}-\frac{u_{h+\epsilon \gamma}^2}{K}-c\frac{u_{h+\epsilon \gamma}^2}{1+u_{h+\epsilon \gamma}^2}-(h+\epsilon \gamma)u_{h+\epsilon \gamma}\right).\\
 \frac{\partial u_{h+\epsilon \gamma}}{\partial \eta}+&qu_{h+\epsilon\gamma}=0
 \end{aligned}
 \right.
\end{equation}
and
\begin{equation}\label{eq:uep1}
\left\{
\begin{split}
-\Delta u_h=&\lambda\left(u_h-\frac{u_h^2}{K}-c\frac{u_h^2}{1+u_h^2}-hu_h\right),\\
\frac{\partial u_h}{\partial\eta}&+qu_{h}=0.
\end{split}
\right.
\end{equation}
Subtracting \eqref{eq:uep1} from \eqref{eq:uep} and dividing by $\epsilon$, we have
\begin{equation}\label{weak}
\left\{
\begin{split}
 -\Delta\frac{u_{h+\epsilon \gamma}-u_h}{\epsilon}&=\lambda\Big(\frac{u_{h+\epsilon \gamma}-u_h}{\epsilon}-\frac{1}{K}\frac{u_{h+\epsilon \gamma}^2-u_{h}^2}{\epsilon}-c\frac{u_{h+\epsilon \gamma}+u_{h}}{(1+u_{h+\epsilon \gamma}^2)(1+u_{h}^2)}\frac{u_{h+\epsilon\gamma}-u_h}{\epsilon}\\
 &-h(x)\frac{u_{h+\epsilon \gamma}-u_h}{\epsilon}-\gamma u_{h+\epsilon \gamma} \Big),\nonumber\\
 \frac{\partial}{\partial\eta}(\frac{u_{h+\epsilon \gamma}-u_h}{\epsilon})&+q\frac{u_{h+\epsilon \gamma}-u_h}{\epsilon}=0.
\end{split}
\right.
\end{equation}
Multiplying both sides by $\frac{u_{h+\epsilon \gamma}-u_h}{\epsilon}$ and integrating over $\Omega$, we obtain
\begin{equation}\label{weak form for Lemma 4.3}
\begin{split}
\int_{\Omega}|\nabla \frac{u_{h+\epsilon \gamma}-u_h}{\epsilon}|^2 ~\rm{d}x+q\int_{\partial\Omega}\left(\frac{u_{h+\epsilon \gamma}-u_h}{\epsilon}\right)^2~\rm{d}s+\\
\int_{\Omega}\lambda\Big(-1+\frac{u_{h+\epsilon\gamma}+u_h}{K}+
c\frac{u_{h+\epsilon\gamma}+u_h}{(1+u_{h+\epsilon \gamma}^2)(1+u_{h}^2)}+h(x)\Big) \Big(\frac{u_{h+\epsilon \gamma}-u_h}{\epsilon}\Big)^2 ~\rm{d}x\\=
\lambda\int_{\Omega}\gamma u_{h+\epsilon \gamma} \frac{u_{h+\epsilon l}-u_h}{\epsilon} ~\rm{d}x.
\end{split}
\end{equation}
From Lemma \ref{newlemma} we have $\sigma_1\left(\lambda(-1+h(x)+\frac{u_{h+\epsilon l}+u_h}{K}+c\frac{u_{h+\epsilon l}+u_h}{(1+u_{h+\epsilon l}^2)(1+u_{h}^2)})\right)>0$. Hence by  \eqref{eig-q1} there exists a constant $M>0$ such that
\begin{eqnarray*}
M  \left\| \frac{u_{h+\epsilon \gamma}-u_h}{\epsilon} \right\|_{H^1(\Omega)}^2 &\leq&\int_{\Omega}|\nabla \frac{u_{h+\epsilon \gamma}-u_h}{\epsilon}|^2~\rm{d}x +\int_{\Omega}\lambda\Big(-1+h+\frac{u_{h+\epsilon \gamma}+u_h}{K}\\
&+&c\frac{u_{h+\epsilon \gamma}+u_h}{(1+u_{h+\epsilon \gamma}^2)(1+u_{h}^2)}\Big) \Big(\frac{u_{h+\epsilon \gamma}-u_h}{\epsilon}\Big)^2 ~\rm{d}x + q\int_{\partial\Omega}\left(\frac{u_{h+\epsilon \gamma}-u_h}{\epsilon}\right)^2~\rm{d}s\\
&=&\lambda \int_{\Omega}\gamma~u_{h+\epsilon \gamma} \frac{u_{h+\epsilon \gamma}-u_h}{\epsilon} ~\rm{d}x\\
&\leq&\lambda \|u_{h+\epsilon \gamma}\|_\infty\Big(\int_{\Omega}\gamma^2~\rm{d}x\Big)^{\frac{1}{2}}\Big(\int_{\Omega}\Big[\frac{u_{h+\epsilon \gamma}-u_h}{\epsilon}\Big]^2\Big)^{\frac{1}{2}}~\rm{d}x\\
&\leq&\lambda \|u_{h+\epsilon \gamma}\|_\infty \Big(\int_{\Omega}\gamma^2~\rm{d}x\Big)^{\frac{1}{2}}\left\| \frac{u_{h+\epsilon \gamma}-u_h}{\epsilon} \right\|_{H^1(\Omega)} \\
&\leq&C_1 \left\| \frac{u_{h+\epsilon \gamma}-u_h}{\epsilon} \right\|_{H^1(\Omega)} ,
\end{eqnarray*}
where $C_1=\lambda\|u_{h+\epsilon l}\|_\infty \|\gamma\|_\infty\|\Omega\|^\frac{1}{2}$.
Dividing both sides by $\left\| \frac{u_{h+\epsilon \gamma}-u_h}{\epsilon} \right\|_{H^1(\Omega)}^2 $ we have  $\|\frac{u_{h+\epsilon \gamma}-u_h}{\epsilon}\|_{H^1(\Omega)}$ is uniformly bounded by $\frac{C_1}{M}$. Thus there exist $\psi\in H^1(\Omega)$ such that $\frac{u_{h+\epsilon \gamma}-u_h}{\epsilon}\rightharpoonup\psi$ in ${H^1(\Omega)}$. We also have $u_{h+\epsilon \gamma}\to u_h$ in $L^2(\Omega)$. Using \eqref{weak form for Lemma 4.3} we get $\psi$ satisfies \eqref{eqn:sensitivity}.  From Lemma \ref{newlemma} we have
$\sigma_1(\Big(\lambda(-1+h(x)+\frac{2u_h}{K}+\frac{2cu_h}{(1+u_{h}^2)^2})\Big)>0$.
Hence by Lemma \ref{lin-q-unique}, the linear problem  \eqref{eqn:sensitivity} has a unique solution.
\end{proof}
\noindent Next, we will prove Theorem 1.3.\\
{\bf{Proof of Theorem 1.3 :}}
By Theorem 2.4.2.7 from \cite{grisvard2011elliptic}, the adjoint problem \eqref{eq:adjoint} has a unique solution $p\in H^2(\Omega).$ 
Now we will derive the characterization of optimal control $h(x)$ in terms of the solution to the adjoint problem. 
This part of the proof follows along the lines of  \cite{D-Lenhart2009}.
Suppose $h(x)$ is an optimal control and $\gamma\in L^\infty(\Omega)$ is such that $h+\epsilon\gamma \in U$ for small $\epsilon>0$.
Then the derivative of $J(h)$ with respect to $h$ in the direction of $\gamma$ satisfies
\begin{eqnarray*}
 0&\geq&\displaystyle{\lim_{\epsilon\to 0^+}}\frac{J(h+\epsilon\gamma)-J(h)}{\epsilon} \nonumber\\
 &=&\displaystyle{\lim_{\epsilon\to 0^+}}\left[\int_{\Omega}\left(\frac{u_{h+\epsilon \gamma}-u_h}{\epsilon} h+\gamma u_{h+\epsilon \gamma}\right) ~\rm{d}x-\int_{\Omega}\left(B_1\gamma+B_2(2h\gamma+\epsilon \gamma^2)\right)~\rm{d}x\right]\nonumber\\
 &=&\int_{\Omega}[\psi h+\gamma u_h]~\rm{d}x-\int_{\Omega}\left(B_1\gamma+2h\gamma B_2\right)~\rm{d}x\nonumber\\
 &=&\int_{\Omega}\Big[\psi\left(-\Delta p -\lambda \Big(p-\frac{2u_hp}{K}-\frac{2cu_h p}{(1+u_h^2)^2}-h(x)p \Big)\right)+\gamma u_h\Big]~\rm{d}x-\int_{\Omega}\left(B_1\gamma+2B_2h\gamma \right)~\rm{d}x,\nonumber
\end{eqnarray*}
where $p$ is the solution to the adjoint problem \eqref{eq:adjoint}. By using \eqref{eqn:sensitivity} we further have
\begin{eqnarray*}
 0&\geq&\int_{\Omega}\Big[\nabla p \cdot\nabla \psi-\lambda \Big(\psi-\frac{2u_h\psi}{K}-2cu_h\frac{\psi}{(1+u_h^2)^2} -h(x)\psi \Big)p+\gamma u_h\Big]~\rm{d}x  + \int_{\partial \Omega} qp\psi ~\rm{d}s -
 \int_{\Omega}\left(B_1\gamma+2B_2h\gamma\right)~\rm{d}x\\
 &=&\int_{\Omega}\Big[-\lambda \gamma u_hp+\gamma u_h\Big]~\rm{d}x-\int_{\Omega}\left(B_1\gamma+2h\gamma B_2\right)~\rm{d}x\\
 &=&\int_{\Omega}\gamma\left(-\lambda u_hp+u_h-B_1-2B_2h\right)~\rm{d}x.
\end{eqnarray*}
Let $D=\{x\in \Omega : 0<h(x)<H\}$. Choosing variations $\gamma$ with support on $D$, the above inequality will be satisfied if and only if
$-\lambda u_hp+u_h-B_1-2B_2h=0$ on $D$. 
Thus the characterization for optimal control can be given in compact form as follows,
\begin{equation}\label{control}
h =\min\Big\{H,\max\{0,\frac{u_h-\lambda pu_h-B_1}{2B_2}\}\Big\}.
\end{equation}
\qed
\begin{remark}
If $B_1=B_2=0$ in $J(h)$, then the optimal control is given by
$$
h(x) = \left\{
        \begin{array}{ll}
            0, & \quad \mbox{if~} p>1 \\
            H, &\quad \mbox{if~} p<1\\
            \frac{\lambda}{2,} & \quad \mbox{if~} p = 1.
        \end{array}
    \right.
$$
Please see Theorem 4.3 in  \cite{D-Lenhart2009} for details.
\end{remark} 
\section{Uniqueness of optimality system}\label{s5}
In this section, we prove that the optimality system has a unique solution. The proof is a modification of the arguments in Ding et al. \cite{D-Lenhart2009}. In order to prove the uniqueness, we need a bound for the adjoint $p$ in $L^{\infty}(\Omega)$ which depends only on $B_2$.  
 \begin{lemma}
Let $\Omega\subset\mathbb{R}^N$, where $N=2,3$ and $B_2\neq 0$. Suppose $u_{h},p,h$ is a solution of \eqref{os} with $u_{h}>0$ in $\Omega$, Then the adjoint $p$ satisfies 
\begin{equation}\label{U1}
\|p\|_{L^{\infty}(\Omega)}\leq\frac{M'}{B_2},
\end{equation} 
where $M'$ does not depend on $B_2$.
\end{lemma}
\begin{proof}
Taking $u_g=u_h$ in Lemma \ref{newlemma}, we have $\sigma_1(\lambda[-1+\frac{2u_h}{K}+\frac{2cu_h}{(1+u_h^2)^2}+h])>0$. 

For $N\in\{2,3\},$ from \cite[Section 4.27]{Adams2003}, we have $H^2(\Omega)\subset \subset C(\overline{\Omega}).$ So there exists $M_0(\Omega,N)$ such that
\begin{equation}
    \label{L_infinity-H2}
    \left\|p \right\|_{L^\infty(\Omega)} \leq M_0\left\|p \right\|_{H^2(\Omega)}.
\end{equation}
We can use Theorem 3.1.2.3 from  \cite{grisvard2011elliptic} in \eqref{eq:adjoint} since $\Omega$ is $C^2$ (see Remark 3.1.2.4 from \cite{grisvard2011elliptic}) to get 
\begin{equation}\label{H2 bound for p}
\begin{aligned}
     \left\|p \right\|_{H^2(\Omega)} &\leq  M_1(\lambda, \Omega)\left\|-\Delta p+\lambda p \right\|_{L^2(\Omega)} \\
     &=M_1\left\|2\lambda p - \frac{2\lambda u_h p}{K} - \frac{2\lambda c u_h p}{(1+u_h^2)^2} - \lambda hp + h \right\|_{L^2(\Omega)}\\
     &\leq M_2 \left\|p \right\|_{L^2(\Omega)}+M_1\left\|h \right\|_{L^2(\Omega)}. \\
\end{aligned}
\end{equation}
Using \eqref{eig-q1}, we have 
\begin{equation*}
\begin{aligned}
     M\left\|p \right\|_{L^2(\Omega)}^2 &\leq  M\left\|p \right\|_{H^1(\Omega)}^2 \\
&\leq \int_{\Omega}|\nabla p|^2~\rm{d}x+\int_{\Omega}\lambda \left ( -1+  \frac{2u_h}{K}+\frac{2c u_h}{(1+u_h^2)^2}+h \right ) p^2~\rm{d}x +\int_{\partial\Omega}qp^2~\rm{d}s \\
&=\int_{\Omega} hp ~\rm{d}x \\
&\leq \left (\int_{\Omega} h^2 ~\rm{d}x \right )^\frac{1}{2} \left (\int_{\Omega} p^2 ~\rm{d}x \right )^\frac{1}{2} \\
&\leq  H|\Omega|^\frac{1}{2}\left (\int_{\Omega} p^2 ~\rm{d}x \right )^\frac{1}{2}.
\end{aligned}
\end{equation*}
So \begin{equation} \label{L2 bound for p without B2}
    \left\|p \right\|_{L^2(\Omega)} \leq \frac{ H|\Omega|^\frac{1}{2}}{M}.
\end{equation}
We also have 
\begin{equation}\label{L2 bound for h}
    \begin{aligned}
        \left\|h\right\|_{L^2(\Omega)}&=\left\|\frac{u_h-\lambda pu_h-B_1}{2B_2}\right\|_{L^2(\Omega)} \\
&\leq\frac{K}{2B_2}|\Omega|^\frac{1}{2}+\frac{\lambda K}{2B_2}\left\|p\right\|_{L^2(\Omega)}+\frac{B_1}{2B_2}|\Omega|^\frac{1}{2} \\ 
&\leq\frac{K}{2B_2}|\Omega|^\frac{1}{2}+\frac{\lambda K H}{2B_2 M}|\Omega|^\frac{1}{2}+\frac{B_1}{2B_2}|\Omega|^\frac{1}{2} (\text{ using } \eqref{L2 bound for p without B2} )\\ 
&=\frac{M_3}{B_2}, \text{ where } M_3:=\frac{|\Omega|^\frac{1}{2}}{2}  \left (K+ \frac{\lambda K H}{M}+B_1 \right ).
    \end{aligned}
\end{equation}
As in the calculations for \eqref{L2 bound for p without B2}, 
\begin{equation*}
\begin{aligned}
     M\left\|p \right\|_{L^2(\Omega)}^2 &\leq \int_{\Omega} hp ~\rm{d}x \\
    &\leq \frac{1}{4\varepsilon}\int_{\Omega} h^2 ~\rm{d}x  +\varepsilon \int_{\Omega} p^2 ~\rm{d}x \\
    &\leq \frac{1}{4\varepsilon}\left ( \frac{M_3}{B_2} \right )^2+\varepsilon \int_{\Omega} p^2 ~\rm{d}x
\end{aligned}
\end{equation*}
Choosing $\varepsilon = \frac{M}{2},$
\begin{equation}\label{L2 bound for p with B2}
    \left\|p \right\|_{L^2(\Omega)} \leq \frac{M_4}{B2}, \text{ where } M_4:=\frac{M_3}{M}.
\end{equation}
Using \eqref{L2 bound for p with B2} and \eqref{L2 bound for h} in \eqref{H2 bound for p}, 
\begin{equation}
    \left\|p \right\|_{H^2(\Omega)} \leq \frac{M_2 M_4 +M_1 M_3}{B_2}.
\end{equation}
Substituting \eqref{H2 bound for p} in \eqref{L_infinity-H2}, we get 
\begin{equation*}
    \left\|p \right\|_{L^\infty(\Omega)} \leq \frac{M'}{B_2}, \text{ where } M':= M_0(M_2 M_4 +M_1 M_3).
\end{equation*}
\end{proof}
\noindent Next we prove Theorem 1.4 for the uniqueness of the optimality system.\\
\noindent {\bf Proof of Theorem 1.4:}
Suppose $u_h,p,h$ and $u_{\bar{h}},\bar{p},\bar{h}$ are two solutions of \eqref{os}. From the characterization of $h$ and $\bar{h}$ given in \eqref{control} we have 
\begin{align*}
|h-\bar{h}|=\left|\frac{u_h-\lambda pu_h-B_1}{2B_2}-\frac{u_{\bar{h}}-\lambda \bar{p}u_{\bar{h}}-B_1}{2B_2}\right |.
\end{align*}
Rewriting $\lambda pu_h-\lambda \bar{p}u_{\bar{h}}=\lambda p(u_h-u_{\bar{h}})+\lambda (p-\bar{p})u_{\bar{h}}$ we have
\begin{align}\label{U7}
|h-\bar{h}|\leq\frac{1}{2B_2}\left(|(1-\lambda p)(u_h-u_{\bar{h}})|+\lambda u_{\bar{h}}|p-\bar{p}|\right).
\end{align}
Choosing test functions $u_h-u_{\bar{h}}$ in the state equation \eqref{main} and using $hu_h-\bar{h}u_{\bar{h}}=h(u_h-u_{\bar{h}})+(h-\bar{h})u_{\bar{h}}$, gives
\begin{align}\label{U8}
\int_{\Omega}|\nabla(u_h-u_{\bar{h}})|^2~\rm{d}x&+q\int_{\partial\Omega}(u_h-u_{\bar{h}})^2~\rm{d}s+\lambda\int_{\Omega}\left(-1+\frac{1}{K}(u_h+u_{\bar{h}})+c\frac{u_h+u_{\bar{h}}}{(1+u_h^2)(1+u_{\bar{h}}^2)}+h\right)(u_h-u_{\bar{h}})^2~\rm{d}x\\
&=\int_{\Omega}\lambda u_{\bar{h}}(h-\bar{h})(u_h-u_{\bar{h}})~\rm{d}x.\nonumber
\end{align}
 Similarly choosing test function $p-\bar{p}$ in the adjoint equation we have  
 \begin{align*}
 \int_{\Omega}|\nabla (p-\bar{p})|^2~\rm{d}x&+q\int_{\partial\Omega}(p-\bar{p})^2~\rm{d}s\\
 &+\lambda\int_{\Omega}\left[-(p-\bar{p})+\frac{2}{K}(u_hp-u_{\bar{h}}\bar{p})+2c\left(\frac{u_hp}{(1+u_h^2)^2}-\frac{u_{\bar{h}}\bar{p}}{(1+u_{\bar{h}}^2)^2}\right)+hp-\bar{h}\bar{p}\right](p-\bar{p})~\rm{d}x\\
 &=\int_{\Omega}(h-\bar{h})(p-\bar{p})~\rm{d}x.
 \end{align*}
Using $u_hp-u_{\bar{h}}\bar{p}=(u_h-u_{\bar{h}})p+u_{\bar{h}}(p-\bar{p})$,  $hp-\bar{h}\bar{p}=(h-\bar{h})p+\bar{h}(p-\bar{p})$ and adding $\frac{u_{\bar{h}}p}{(1+u_{\bar{h}}^2)^2}$ both side of the above equation we have 
 \begin{align}\label{U9}
 \int_{\Omega}|\nabla (p-\bar{p})|^2~\rm{d}x&+q\int_{\partial\Omega}(p-\bar{p})^2~\rm{d}s+\lambda\int_{\Omega}\left(-1+\frac{2}{K}u_{\bar{h}}+2c\frac{u_{\bar{h}}}{(1+u_{\bar{h}}^2)^2}+\bar{h}\right)(p-\bar{p})^2~\rm{d}x\\\nonumber
 &=\int_{\Omega}(h-\bar{h})(p-\bar{p})~\rm{d}x-\frac{2\lambda}{K}\int_{\Omega}p(u_h-u_{\bar{h}})(p-\bar{p})~\rm{d}x\\
 &-\lambda\int_{\Omega}p(h-\bar{h})(p-\bar{p})~\rm{d}x+2c\lambda\int_{\Omega}p\left(\frac{u_{\bar{h}}}{(1+u_{\bar{h}}^2)^2}-\frac{u_h}{(1+u_h^2)^2}\right)(p-\bar{p})~\rm{d}x\nonumber
 \end{align}
Add \eqref{U8} and \eqref{U9}, to get,
\begin{align*}
&\int_{\Omega}|\nabla(u_h-u_{\bar{h}})|^2~\rm{d}x+q\int_{\partial\Omega}(u_h-u_{\bar{h}})^2~\rm{d}s+\int_{\Omega}\lambda\left(-1+\frac{1}{K}(u_h+u_{\bar{h}})+c\frac{u_h+u_{\bar{h}}}{(1+u_h^2)(1+{u_{\bar{h}}^2})}+h\right)(u_h-u_{\bar{h}})^2~\rm{d}x\\\nonumber
&\int_{\Omega}|\nabla (p-\bar{p})|^2~\rm{d}x+q\int_{\partial\Omega}(p-\bar{p})^2~\rm{d}s+\lambda\int_{\Omega}\left(-1+\frac{2}{K}u_{\bar{h}}+2c\frac{u_{\bar{h}}}{(1+u_{\bar{h}}^2)^2}+\bar{h}\right)(p-\bar{p})^2~\rm{d}x\\\nonumber
&~~~~~=\int_{\Omega}\lambda u_{\bar{h}}(h-\bar{h})(u_h-u_{\bar{h}})~\rm{d}x\\\nonumber
+&\int_{\Omega}(h-\bar{h})(p-\bar{p})~\rm{d}x-\frac{2\lambda}{K}\int_{\Omega}p(u_h-u_{\bar{h}})(p-\bar{p})~\rm{d}x\\\nonumber
 &-\lambda\int_{\Omega}p(h-\bar{h})(p-\bar{p})~\rm{d}x+2c\lambda\int_{\Omega}p\left(\frac{u_{\bar{h}}}{(1+u_{\bar{h}}^2)^2}-\frac{u_h}{(1+u_h^2)^2}\right)(p-\bar{p})~\rm{d}x.\nonumber
\end{align*}
Because $u_h,u_{\bar{h}}>0$, and $u_h,u_{\bar{h}}$ satisfy the state equation \eqref{os}, by Lemma \ref{newlemma} we have
\begin{align*}
\sigma_1\left(\lambda\left(-1+\frac{1}{K}(u_h+u_{\bar{h}})+c\frac{u_h+u_{\bar{h}}}{(1+u_h^2)(1+{u_{\bar{h}}}^2)}+h\right)\right)>0.
\end{align*}
From this we also get
\begin{align*}
\sigma_1\left(\lambda \left(-1+\frac{2}{K}u_{\bar{h}}+2c\frac{u_{\bar{h}}}{(1+u_{\bar{h}}^2)^2}+\bar{h}\right)\right)>0.
\end{align*}
Hence by $\eqref{eig-q1}$ there exists $M>0$, such that 
\begin{align*}
&M \left ( \left\| u_h-u_{\bar{h}} \right\|_{H^1(\Omega)}^2 +\left\| p-\bar{p} \right\|_{H^1(\Omega)}^2\right )\\
&\leq\int_{\Omega}|\nabla(u_h-u_{\bar{h}})|^2~\rm{d}x+q\int_{\partial\Omega}(u_h-u_{\bar{h}})^2~\rm{d}s+\int_{\Omega}\lambda\left(-1+\frac{1}{K}(u_h+u_{\bar{h}})+c\frac{u_h+u_{\bar{h}}}{(1+u_h^2)(1+u_{\bar{h}}^2)}+h\right)(u_h-u_{\bar{h}})^2~\rm{d}x\\
&+\int_{\Omega}|\nabla (p-\bar{p})|^2~\rm{d}x+q\int_{\partial\Omega}(p-\bar{p})^2~\rm{d}s+\lambda\int_{\Omega}\left(-1+\frac{2}{K}u_{\bar{h}}+2c\frac{u_{\bar{h}}}{(1+u_{\bar{h}}^2)^2}+\bar{h}\right)(p-\bar{p})~\rm{d}x\\
&=\int_{\Omega}\lambda u_{\bar{h}}(h-\bar{h})(u_h-u_{\bar{h}})~\rm{d}x
+\int_{\Omega}(h-\bar{h})(p-\bar{p})~\rm{d}x-\frac{2\lambda}{K}\int_{\Omega}p(u_h-u_{\bar{h}})(p-\bar{p})~\rm{d}x\\
 &-\lambda\int_{\Omega}p(h-\bar{h})(p-\bar{p})~\rm{d}x+2c\lambda\int_{\Omega}p\left(\frac{u_{\bar{h}}}{(1+u_{\bar{h}}^2)^2}-\frac{u_h}{(1+u_h^2)^2}\right)(p-\bar{p})~\rm{d}x\\
 &=\int_{\Omega}\lambda u_{\bar{h}}(h-\bar{h})(u_h-u_{\bar{h}})~\rm{d}x
+\int_{\Omega}(h-\bar{h})(p-\bar{p})~\rm{d}x-\frac{2\lambda}{K}\int_{\Omega}p(u_h-u_{\bar{h}})(p-\bar{p})~\rm{d}x\\
 &-\lambda\int_{\Omega}p(h-\bar{h})(p-\bar{p})~\rm{d}x+2c\lambda\int_{\Omega}p\left(\frac{-1+2u_hu_{\bar{h}}+u_h^3u_{\bar{h}}+u_h^2u_{\bar{h}}^2+u_hu_{\bar{h}}^3}{(1+u_h^2)^2(1+u_{\bar{h}}^2)^2}\right)(u_h-u_{\bar{h}})(p-\bar{p})~\rm{d}x.
 \end{align*}
 Using \eqref{U1}, \eqref{U7}, and $\frac{K(1-H)}{2}\leq u_h,u_{\bar{h}}\leq K$ in the above inequality we have
 \begin{align*}
 &M \left ( \left\| u_h-u_{\bar{h}} \right\|_{H^1(\Omega)}^2 +\left\| p-\bar{p} \right\|_{H^1(\Omega)}^2\right )\\
&\leq \frac{\lambda K\|1-\lambda p\|_{L^{\infty}}}{2B_2}\int_{\Omega}(u_h-u_{\bar{h}})^2~\rm{d}x+\frac{(\lambda K)^2}{2B_2}\int_{\Omega}|u_h-u_{\bar{h}}||p-\bar{p}|~\rm{d}x+\frac{2\lambda M'}{KB_2}\int_{\Omega}|u_h-u_{\bar{h}}||(p-\bar{p})|~\rm{d}x\\
 &+\frac{1+\lambda\|p\|_{L^{\infty}}}{2B_2}\left\{\|1-\lambda p\|_{L^{\infty}}\int_{\Omega}|u_h-u_{\bar{h}}||p-\bar{p}|~\rm{d}x+\int_{\Omega}(p-\bar{p})^2~\rm{d}x\right\}+\frac{2c\lambda M'M_5(K,H)}{B_2}\int_{\Omega}|u_h-u_{\bar{h}}||p-\bar{p}|~\rm{d}x.
\end{align*}
Using the Cauchy's inequality and \eqref{U1}, we have 
\begin{align*}
&M \left ( \left\| u_h-u_{\bar{h}} \right\|_{H^1(\Omega)}^2 +\left\| p-\bar{p} \right\|_{H^1(\Omega)}^2\right )\\
&\leq \frac{M_6}{B_2}\left ( \left\| u_h-u_{\bar{h}} \right\|_{H^1(\Omega)}^2 +\left\| p-\bar{p} \right\|_{H^1(\Omega)}^2\right ).
\end{align*}
This leads to a contradiction if we choose $B_2$ sufficiently large. Thus we conclude that the optimal control is unique for large $B_2$.


\bibliography{OptimalRef.bib}
\bibliographystyle{abbrv}
\end{document}